\documentclass[a4paper,12pt,reqno]{amsart}
\usepackage[latin1]{inputenc}
\usepackage[english]{babel}
\usepackage{amssymb,times}
\usepackage{graphicx}
\usepackage[pdftex,usenames,dvipsnames]{color}
\usepackage{xcolor}
\usepackage[all]{xy}
\usepackage[bookmarks=true,hyperindex,pdftex,colorlinks, citecolor=blue]{hyperref}
\usepackage{url}
\usepackage{mathtools}  

\usepackage[mathscr]{euscript}

\newcommand{\aj}[1]{\textcolor{blue}{#1}}

\renewcommand{\theenumi}{(\roman{enumi})}
\renewcommand{\labelenumi}{(\roman{enumi})}

\newcommand{\T}{\mathbb{T}}

\newcommand{\K}{\mathbb{K}}

\newcommand{\C}{\mathbb{C}}
\newcommand{\R}{\mathbb{R}}
\newcommand{\N}{\mathbb{N}}
\newcommand{\G}{\Gamma}
\newcommand\eps{\ensuremath{\varepsilon}}
\newcommand{\bof}{\mathrm{bof}}

\newcommand{\norm}[1]{\left\Vert#1\right\Vert}

\newcommand{\dist}{\mathrm{dist}}
\newcommand{\diam}{\mathrm{diam}}

\newcommand{\aconv}{{\mathrm{aco}}}
\newcommand{\itp}{\mathbin{\hat{\otimes}_\varepsilon}}

\newcommand{\re}{\operatorname{Re}}
\newcommand{\st}{\operatorname{St}}
\newcommand{\Fl}{\operatorname{Fl_\G}}
\newcommand\ACK{\mathrm{ACK}}
\newcommand{\algebra}{\ensuremath{\mathscr{A}}}

\newtheorem{theo}{Theorem}[section]
\newtheorem{lem}[theo]{Lemma}
\newtheorem{pro}[theo]{Proposition}
\newtheorem{exa}{Example}

\newtheorem{cor}[theo]{Corollary}

\newtheorem{rem}[theo]{Remark}

\theoremstyle{definition}
\newtheorem{defi}[theo]{Definition}

\theoremstyle{remark}

\numberwithin{equation}{section}

\title[$\G$-flatness and Bishop--Phelps--Bollob\'as type theorems]{$\G$-flatness and Bishop--Phelps--Bollob\'as type theorems for operators}

\date{VERSION: 31 December, 2016}

\begin{document}

\author[Cascales]{Bernardo Cascales}
\author[Guirao]{Antonio J. Guirao}
\author[Kadets]{Vladimir Kadets}
\author[Soloviova]{Mariia Soloviova}

\thanks{The research of first, second and third authors was partially supported by
MINECO grant MTM2014-57838-C2-1-P and Fundaci\'on S\'eneca, Regi\'on de Murcia  grant 19368/PI/14. The research of the third author is done in frames of Ukranian Ministry of Science and Education Research Program $0115$U$000481$. The  research of fourth author has been partially performed during her stay in University of Murcia in frames of Erasmus+ program.}

\address{Depto de Matem\'aticas, Universidad de Murcia, 30100 Espinardo, Murcia, Spain}\email{beca@um.es}

\address{IUMPA, Universitat Polit\`ectnica de Val\`encia, 46022, Valencia, Spain}\email{anguisa2@mat.upv.es}

\address{Department of Mathematics and Informatics\\ V.N. Karazin Kharkiv  National University\\ 61022 Kharkiv, Ukraine}
\email{v.kateds@karazin.ua}
\email{mariiasoloviova93@gmail.com}


\begin{abstract}
The  Bishop--Phelps--Bollob\'{a}s property deals with simultaneous approximation of an operator $T$ and a vector $x$ at which $T$ nearly attains its norm by an operator $T_0$ and a vector $x_0$, respectively, such that $T_0$ attains its norm at $x_0$. In this note we extend the already known results about {the} Bishop--Phelps--Bollob\'{a}s property for Asplund operators  to a wider class of Banach spaces and to a wider class of operators. Instead of proving a BPB-type theorem for each space separately we isolate two main notions: $\G$-flat operators and Banach spaces with $\ACK_\rho$ structure. In particular, we prove a general  BPB-type theorem for $\G$-flat operators acting to a space with $\ACK_\rho$ structure and show that uniform algebras and spaces with the property $\beta$ have $\ACK_\rho$ structure.  We also study the stability of the $\ACK_\rho$ structure under some natural Banach space theory operations. As a consequence, we discover many new examples of spaces $Y$ such that the Bishop--Phelps--Bollob\'{a}s property for Asplund operators is valid for all pairs of the form ($X,Y$).
\end{abstract}
\maketitle

\section{Introduction}

In this paper $X$, $Y$ are Banach spaces (real or complex), $\K$ stands for the field of scalars $\R$ or $\C$,  $L(X,Y)$ is the space of all bounded linear operators $T\colon X \to Y$, $L(X) = L(X,X)$,  $B_X$ and $S_X$ denote the closed unit ball and the unit sphere of $X$, respectively and $\aconv A$ stands for the absolute convex hull of the set $A$.

According to  \cite{aco-aro-gar-mae}, a pair $(X,Y)$ has the \emph{Bishop--Phelps--Bollob\'as property} (BPB property) for operators if for {every} $\eps>0$ there exists $\delta(\eps)>0$ such that for every  operator  $T\in L(X,Y)$ of norm 1, if $x_0\in S_X$ is such that $\norm{T(x_0)}>1-\delta(\eps)$, then there exist $u_0\in S_X$ and $S\in S_{L(X,Y)}$ satisfying $ \norm{S(u_0)}=1$, $ \norm{x_0-u_0}<\eps$, and $\norm{T-S}<\eps$.

If an analogous definition is valid for operators $T$, $S$ from a subspace  $\mathcal I \subset L(X,Y)$, then we say that  $(X,Y)$ has the \emph{Bishop--Phelps--Bollob\'as property for operators from} $\mathcal I$.

With this terminology, the original Bishop--Phelps--Bollob\'as theorem \cite{bol} says that for every $X$, the pair $(X, \K)$ has the BPB property for operators. Also, thanks to Acosta, Aron,  Garc\'{\i}a, and Maestre  \cite[Theorem 2.2]{aco-aro-gar-mae}, if $Y$ has the Lindenstrauss' property $\beta$ (Definition \ref{def:beta}), then for every Banach space  $X$ the pair $(X, Y)$ has the Bishop--Phelps--Bollob\'{a}s property for operators.

In 2011 Aron, Cascales, and Kozhushkina  \cite[Theorem 2.4]{cas-alt23} showed that for every $X$ and every compact Hausdorff space $K$  the pair $(X, C(K))$  has the  BPB property for Asplund operators  (Definition \ref{def_aspl-oper}).  In 2013 Cascales, Guirao and Kadets \cite{CasGuiKad} extended this result to uniform algebras $ \algebra \subset C(K)$. The exact statement of the last result is given below.

\begin{theo}[{\cite[Theorem 3.6]{CasGuiKad}}]\label{theo:thetheoremainFA}
Let $ \algebra \subset C(K)$ be a uniform algebra and  $T\colon X\to  \algebra$ be an Asplund operator with $\norm{T}=1$. Suppose that {$0<\eps<\sqrt{2}$} and $x_0\in S_X$ are such that $\norm{Tx_0}>1-\frac{\eps^2}{2}$. Then there exist $u_0\in S_X$ and an Asplund operator $S\in S_{L(X, \algebra)}$ satisfying that:
\[
\norm{Su_0}=1,\quad \norm{x_0-u_0} \leq \eps\quad\text{ and }\quad\norm{T-S}<2\eps.
\]
\end{theo}

In the same vein, Acosta,  Becerra Guerrero,  Garc\'{\i}a, Kim, and  Maestre \cite{aco-guer-gar-kim-mae} generalized \cite[Theorem 2.4]{cas-alt23} to some spaces of continuous vector-valued functions (see Theorem \ref{theo:C(K,Y)} below).

The aim of this paper is to extend all these results to a wider class of Banach spaces and to a wider class of operators.  The main difference of our approach is that instead of proving a Bishop--Phelps--Bollob\'{a}s kind theorem for each space separately (and thus repeating essential parts of the proof many times), we introduce a new Banach space property (called $\ACK_\rho$ structure) which extracts all the useful technicalities for the  BPB type of approximation. We prove a general  Bishop--Phelps--Bollob\'{a}s type theorem for  $\G$-flat operators (see Definition~\ref{def_gam-flat-op}) acting to a space with $\ACK_\rho$ structure and show that uniform algebras and spaces with {the} property $\beta$ have $\ACK_\rho$ structure.  After that, we study the stability of the $\ACK_\rho$ structure under some natural Banach space theory operations which as a consequence gives us a wide collection of examples of pairs $(X,Y)$ possessing the  BPB property for Asplund operators.

The structure of the paper is as follows. In section $2$ we collect the necessary definitions (in particular that of Asplund operators and of $\G$-flat operators) and prove an important Basic Lemma. In section $3$ we introduce the central concept of $\ACK_\rho$ structure and prove a general BPB type theorem for this class of Banach spaces. Finally, in section $4$ we perform the announced study of spaces with $\ACK_\rho$ structure which, on the one hand, gives a unified proof of several results from  \cite{aco-aro-gar-mae,aco-guer-gar-kim-mae,cas-alt23} and \cite{CasGuiKad}, and on the other hand,  leads  to new  BPB type theorems in concrete spaces.

For the non-defined notions used through this article, we refer to \cite{LosCinco}.


\section{$\G$-Flat operators and the Basic Lemma}

Let $(B, \tau)$ be a topological space,  $\rho$ be a metric on $B$ (possibly, not related with $\tau$). $B$ is said to be \emph{fragmented by}  $\rho$, if for every non-empty subset $A \subset B$ and for every  $\eps>0$ there exists a   $\tau$-open  $U$ such that $U \cap A \neq \emptyset$ and $\diam(U \cap A) < \eps$. Some important examples of fragmented topological spaces  come from Banach space theory.  For instance, every weakly compact subset of a Banach space is fragmented by the norm (i.e., by the metric $\rho(x,y) = \norm{x - y}$), see~\cite{Namioka}.

A Banach space $X$ is called an \textit{Asplund space} if, whenever $f$ is a convex continuous function defined on an open subset $U$ of $X$, the set of all points of $U$ where $f$ is Fr\'{e}chet differentiable is a dense $G_\delta$-subset of $U$. This definition is due to Asplund \cite{Asplund} under the name \textit{strong differentiability space}. This concept has multiple characterizations via topology or measure theory, as in the following:

\begin{theo}[{\cite{NamPh,Steg1,Steg2}}]\label{Aspl-eqviv}
   Let $X$ be a Banach space. Then the following conditions are equivalent:
   \begin{enumerate}
      \item X is an Asplund space;
      \item every $w^*$-compact subset of $(X, w^*)$ is fragmented by the norm;
      \item each separable subspace of $X$ has separable dual;
      \item $X^*$ has {the} Radon-Nikod\'{y}m property.
   \end{enumerate}
\end{theo}

According to the above, every reflexive space and every separable space whose dual is separable is an Asplund space. Classical example of Asplund spaces are  $L_p$ and $\ell_p$ with $1 < p < \infty$, and also $c_0$;  examples of spaces that are not Asplund are $C[0,1]$, $\ell_1$, $\ell_\infty$, $L_1[0,1]$ and  $L_\infty[0,1]$, see for instance~\cite{DU}.

\begin{defi}[{\cite{Steg3}}] \label{def_aspl-oper}
   An operator $T\in L(X,Y)$ is said to be an \textit{Asplund operator} if it factors through an Asplund space, i.e., there exist an Asplund Banach space $Z$ and operators $T_1\in L(X,Z), T_2\in L(Z,Y)$ such that $T=T_2\circ T_1$.
\end{defi}

Compact and weakly compact operators are Asplund operators (every weakly compact operator factorizes through a reflexive space).

Theorem \ref{Aspl-eqviv} yields the following result:

\begin{rem}[{\cite{Steg3}}]\label{Aspl-frag}
    If $T$ is an Asplund operator, then its adjoint $T^*$ sends the unit ball of $Y^*$ into a $w^*$-compact subset of $(X,w^*)$ that is norm fragmented.
\end{rem}

\begin{defi} \label{def_BPBp}
   Let $Y$ be a Banach space. $Y$ is said to have the \emph{Bishop--Phelps--Bollob\'{a}s property for Asplund operators} (A-BPBp for short) if for {every} $\eps>0$ there exists $\delta(\eps)>0$, such that for every Banach space $X$ and every Asplund operator $T\in S_{L(X,Y)}$, if $x_0\in S_X$ is such that $\norm{T(x_0)}>1-\delta(\eps)$, then there exist $u_0\in S_X$ and $S\in S_{L(X,Y)}$ satisfying
    \[
        \norm{S(u_0)}=1, \norm{x_0-u_0}<\eps \text{ and }\norm{T-S}<\eps.
    \]
\end{defi}

\begin{defi}[\cite{Kemp}] \label{def:quasi-cont}
   Let $A$ and $B$ be  topological spaces. A function ${f\colon A \to B}$ is said to be \emph{quasi-continuous}, if for every non-empty open subset $U \subset A$, every $z \in U$ and every neighborhood $V$ of $f(z)$ there exists a non-empty open subset $W\subset U$ such that $f(W)\subset V$.
\end{defi}

Let us introduce some new terminology. Note that a similar concept of fragmentability of maps was introduced in~\cite{Kou}.

\begin{defi} \label{def_op-fragm-func}
   Let $A$ be a topological space and $(M,d)$ be a metric space. A function $f\colon A \to M$ is said to be \emph{openly fragmented}, if for every non-empty open subset $U \subset A$ and every $\eps > 0$ there exists a non-empty open subset $V \subset U$ with $d$-$\diam(f(V)) < \eps$.\end{defi}

Every continuous or quasi-continuous function  $f\colon A \to M$ is openly fragmented. In particular, if $A$ is a discrete topological space then every $f\colon A \to M$ is openly fragmented. For every metric space $M$, every left-continuous  $f\colon [0, 1] \to M$ and every right-continuous function $f\colon [0, 1] \to M$ are openly fragmented. Every  $f\colon A \to M$ with a dense set of continuity points  is openly fragmented. Every separately continuous function of two variables $f\colon [0, 1] \times [0,1] \to M$ is quasi-continuous \cite{baire} and, consequently, openly fragmented.  Some other easy but useful examples are given in  the following theorem:

\begin{theo}\label{theo:op-fragm-comp}
   Let $A, B$ be topological spaces,  $\rho$ be a metric on $B$ (possibly, not related with the original topology), and $f\colon A \to B$ be a function.
   \begin{enumerate}
      \item[(i)] If $B$ is fragmented by  $\rho$, and $f$ is continuous in the original topologies,  then $f\colon  A \to (B, \rho)$ is openly fragmented.
      \item[(ii)] If $A$ is fragmented  by some metric $\rho_1$ and $f\colon (A, \rho_1) \to (B, \rho)$ is uniformly continuous, then $f\colon A \to (B, \rho)$  is openly fragmented.
   \end{enumerate}
   Let, moreover, $(B,\norm{\cdot})$ be a Banach space. Then
   \begin{enumerate}
      \item[(iii)] If $f, g\colon  A \to (B, \norm{\cdot})$ are openly fragmented then $f + g\colon  A \to (B,  \norm{\cdot})$ is openly fragmented.
      \item[(iv)] If $f\colon  A \to (B,  \norm{\cdot})$ and $ g\colon  A \to \K$ are openly fragmented then $gf\colon  A \to (B,  \norm{\cdot})$ is openly fragmented.
   \end{enumerate}
\end{theo}

\begin{proof}
(i) For  {a given non-empty} open subset $U \subset A$ consider $f(U) \subset B$. By $\rho$-fragmentability of {$B$}, for every  $\eps > 0$ there exits an open subset $W$ of  $B$ with $f(U) \cap W \neq \emptyset$ and $\diam(f(U) \cap W) < \eps$. By continuity of $f$ the set $f^{-1}(W)$ is open and $V \coloneqq f^{-1}(W) \cap U$ will be the {non-empty} open subset $V \subset U$   we need.

The statements (ii), (iii) and (iv) are routine.
\end{proof}

\begin{defi} \label{def_gam-flat-op}
   Let $X$, $Y$ be Banach spaces and $\G \subset Y^*$. An operator   $T \in L(X,Y)$   is said to be \emph{$\G$-flat}, if $T^*|_\G\colon (\G, w^*) \to (X^*, \norm{\cdot}_{X^*})$ is openly fragmented. In other words, for every  $w^*$-open subset $U \subset Y^*$ with $U \cap \G \neq \emptyset$ and every  $\eps > 0$ there exists  a  $w^*$-open subset $V \subset U$ with $V \cap \G \neq \emptyset$ such that $\diam (T^*(V \cap \G)) < \eps$.  The set of all $\G$-flat operators in $L(X,Y)$ will be denoted by  $\Fl(X,Y)$.
\end{defi}

Statements (iii) and (iv) of the previous theorem imply that $\Fl(X,Y)$ is a linear subspace of $L(X,Y)$. Let us list some examples of $\G$-flat operators.

\begin{exa}\label{example-asplund-are-G-Flat}
   Every Asplund operator $T \in L(X,Y)$ is $\G$-flat for every $\G \subset B_{Y^*}$. This follows from Remark \ref{Aspl-frag} and Theorem \ref{theo:op-fragm-comp},  (i).
\end{exa}

\begin{exa}\label{example-NFragmented-are-G-Flat}
   If  $(\G, w^*) \subset {Y^*}$ is norm fragmented, then every bounded operator in $L(X,Y)$ is $\G$-flat (Theorem \ref{theo:op-fragm-comp},  (ii)). In particular, we have the next concrete example.
\end{exa}

\begin{exa}\label{example-Discrete-are-G-Flat}
   If  $(\G, w^*) \subset {Y^*}$ is discrete,  then every operator $T \in L(X,Y)$ is $\G$-flat.
\end{exa}

The notion of $\G$-flat generalizes the property of Asplund operators that allowed to prove \cite[Lemma 2.3]{cas-alt23}. The immediate generalization of that lemma is the following result:

\begin{lem}[Basic Lemma]\label{ACK-Lem}
   Let $X, Y$ be Banach spaces, $\Gamma\subset B_{Y^*}$ be a $1$-norming set,  $T \in \Fl(X,Y)$ be a $\G$-flat operator with $\norm{T}=1$ , $0<\eps<2/3$, and $x_0\in S_X$ be such that $\norm{Tx_0} > 1-\eps$. Then for every $r>0$ and for every $k \in [\frac{\eps}{2(1-\eps)}, 1)$ there exist:
   \begin{enumerate}
      \item a $w^*$-open set $U_r \subset Y^*$ with $U_r \cap \G \neq\emptyset$, and
      \item points $x_r^*\in S_{X^*}$ and $u_r\in S_X$ with $|x_r^*(u_r)|=1$ such that
         \begin{equation}\label{eq:newACKlemma}
            \norm{x_0-u_r}\leq \frac{\eps}{k} \ \ \ \text{and} \ \ \
            \norm{T^*z^* -x_r^*}\leq  r + 2k \ \text{ for every } z^*\in U_r \cap \G.
         \end{equation}
   \end{enumerate}
\end{lem}

The proof of this fact is a modification of that of {\cite[Lemma 2.3]{cas-alt23}}. First, we use the following fact:

\begin{pro}[{\cite[Corollary~2.2]{Phelps}}] \label{Corollary2.2Phelps}
   Let $X$ be a real Banach space, $z^*\in S_{X^*}$, $z\in S_X$, $\eta > 0$ and $z^*(z)\geq 1-\eta$. Then for {every} $k\in (0,1)$ there exist $y^*\in S_{X^*}$ and $u\in S_X$ such that
   \begin{equation*}
      y^*(u) = 1,\qquad \norm{z-u} \le \frac{\eta}{k}, \qquad \norm{z^*-{y}^*} \le  2k.
   \end{equation*}
\end{pro}

In the next proposition, we relax the condition $z^*\in S_X$ allowing $\norm{z^*}$ to be smaller than $1$. Note that $x^*$ plays the role of $z^*$.

\begin{pro}\label{SharpBPB}
   Let $X$ be a  Banach space,  $\eps\in (0,2/3)$, $x\in S_X, x^*\in B_{X^*}$ and $|x^*(x)| \geq 1- \eps$. Then, for {every}  $k \in [\frac{\eps}{2(1-\eps)}, 1)$  there exist ${y}^*\in S_{X^*}$ and ${u} \in S_X$ such that
   \begin{equation*}
      |{y}^*({u})| = 1,\qquad \norm{x-{u}} \le \frac{\eps}{k}, \qquad \norm{x^*-{y}^*} \le  2k.
   \end{equation*}
\end{pro}

\begin{proof}
   Without loss of generality we can assume that $x^*(x)\geq 1-\eps$. Then $\norm{x^*} \ge 1 - \eps$. Set $z^*\coloneqq x^*/\norm{x^*}, z\coloneqq x$. Then $z^*(z)\geq 1- \eta$ for
   $\eta=1-(1-\eps)\norm{x^*}^{-1}\in [0, \eps]$. If $\eta = 0$, then $z^*(z) = 1$, so we can take $y^*=z^*$ and $u = x$, which satisfy the inequalities we want. So we may assume that $0 < \eta \le \eps$. Set
   $k_0\coloneqq \frac{k\eta}{\eps} \in (0, 1)$. So, according to Proposition \ref{Corollary2.2Phelps}, there exist $y^*\in S_{X^*}$ and $u\in S_X$ such that
   \begin{equation*}
      y^*(u) = 1,\qquad \norm{z-u} \le \frac{\eta}{k_0}, \qquad \norm{z^*-{y}^*} \le  2k_0.
   \end{equation*}

   Therefore, $\norm{x-u}\leq \eta/k_0=\eps/k$. Also, we have
      \begin{align*}
         \norm{x^*-y^*} &\leq \norm{x^*-z^*}+\norm{z^*-y^*}\leq \norm{x^*-\frac{x^*}{\norm{x^*}}} + 2k_0\\
                        &= 1-\norm{x^*} +2k_0 = 1- \norm{x^*} + \frac{2k}{\eps}  \left( 1- \frac{1-\eps}{\norm{x^*}}\right).
      \end{align*}
   Observe that the function $\psi(t)= 1- t + \frac{2k}{\eps}( 1- \frac{1-\eps}{t} )$ is  increasing when $t\in \left(0, \sqrt{\frac{2k(1-\eps)}{\eps}}\right)$. So, if $k\geq \frac{\eps}{2(1-\eps)}$, we have $\psi(\norm{x^*})\leq \psi(1)=2k$. In this case, we get our conclusion.
\end{proof}

\begin{proof}[Proof of Lemma \ref{ACK-Lem}]
   Use that $\Gamma \subset B_{Y^*}$ is 1-norming and pick $y_0^* \in \Gamma$ such that
   \[
      |T^*(y_0^*)(x_0)|=|y_0^*(Tx_0)|> 1-\eps.
   \]
   Set $U\coloneqq \{ y^*\in Y^* : |T^*y^*(x_0)| > 1-\eps\}$. We have that $y_0^* \in U \cap \G \subset B_{Y^*}$. Since $U$ is $w^*$-open in $Y^*$ and  $U \cap \G \neq \emptyset$, according to Definition \ref{def_gam-flat-op}, for every $r>0$ there exists  a  $w^*$-open subset $U_r \subset U$ with $U_r \cap \G \neq \emptyset$ such that $\diam (T^*(U_r \cap \G)) < r$.

   Fix some $y_1^*\in U_r \cap \G$ and set $x_1^*=T^*y_1^*$. Then, $1\geq \norm{x_1^*} \geq |x_1^*(x_0)| > 1-\eps$ which, by applying Proposition \ref{SharpBPB} {to} any $\frac{\eps}{2(1-\eps)}\leq k< 1$, gives $x^*_r\in S_{X^*}$ and $u_r\in S_X$ with $|x_r^*(u_r)|=1$ and such that
   \[
      \norm{x_0-u_r}\leq\frac{\eps}{k}\quad \text{ and }\quad \norm{x_1^*-x_r^*}\leq 2k.
   \]
   Finally, let $z^*\in U_r\cap \G$ be arbitrary. Then,
   \begin{align*}
      \norm{T^*z^*-x_r^*}\leq \norm{T^*z^*-x_1^*}+\norm{x_1^* - x_r^*}\le r+2k,
   \end{align*}
   which finishes the proof.
\end{proof}


\section{The $\ACK$ structure}

In the definition below we extract the structural properties of $C(K)$ and its uniform subalgebras that were essential in the proof of \cite[Th. 3.6]{CasGuiKad}. The name ``$\ACK$ structure'' comes from the words ``Asplund'' and ``$C(K)$''.

\begin{defi} \label{def_ACK}
Let $X$ be a Banach space and $\mathcal{O}$ be a non-emtpy subset of $L(X)$. We will say that $X$ has \emph{$\mathcal{O}$-ACK structure with parameter $\rho$}, for some $\rho\in[0,1)$ ($X\in\mathcal{O}$-$\ACK_\rho$, for short) whenever there exists a $1$-norming set $\G\subset B_{X^*}$ such that for every $\eps>0$ and every non-empty relatively $w^*$-open subset $U\subset\G$ there exist a non-empty subset $V\subset U$, vectors $x_1^*\in V$, $e\in S_X$ and an operator $F\in\mathcal{O}$ with the following properties:

{\renewcommand{\labelenumi}{\theenumi}
\renewcommand{\theenumi}{(\Roman{enumi})}%
\begin{enumerate}
\item $\norm{Fe} = \norm{F} = 1$;
\item $x_1^*(Fe) = 1$;
\item $F^*x_1^* = x_1^*$;
\item denoting $V_1 = \{x^* \in \G :  \norm{F^*x^*} + (1 - \eps)\norm{(I_{X^*} - F^*)(x^*)} \le 1 \}$, then  $|v^*(Fe)| \le  \rho$ for every $x^* \in \G \setminus V_1$;
\item $\dist(F^*x^*, \aconv\{0, V\}) < \eps$ for every $x^* \in \G$; and \item  $|v^*(e) - 1| \le  \eps$ for every $v^* \in V$.
\end{enumerate}}

The Banach space $X$ is said to have \emph{simple $\mathcal{O}$-ACK structure}  ($X \in \mathcal{O}$-$\ACK$) if $V_1 = \G$. In other words, for $X \in \mathcal{O}$-$\ACK$ the above definition holds true with the following modification:  the property (IV) becomes
\begin{enumerate}
\item[(IV)']  $ \norm{F^*x^*}  + (1 - \eps)\norm{(I_{X^*} - F^*)(x^*)} \le 1$  for every $x^* \in \G$.
\end{enumerate}
In case of $\mathcal{O}=L(X)$, we will simply {say} $\ACK_\rho$ (and simple $ACK$) structure.
\end{defi}

\begin{rem}\label{rem:ACK-rho-inheritance}
If $X$ belongs to the class $\ACK_\rho$, then $X$ also belongs to $\ACK_\sigma$ for every $\sigma\in[\rho, 1)$. Moreover,  $\ACK\subset\ACK_\rho$ for every  $\rho \in [0, 1)$.
\end{rem}

\begin{defi} \label{def_Gflat-ideal}
A linear subspace  $\mathcal I \subset L(X,Y)$ is said to be a \emph{ $\G$-flat ideal}, if  all elements of  $\mathcal I$ are  $\G$-flat operators,  $\mathcal I$ contains all operators of finite rank, and for every $T \in \mathcal I$ and every $F \in L(Y)$ their composition $F \circ T$ belongs to $\mathcal I$.
 \end{defi}

Observe that the subspace of Asplund operators in $L(X,Y)$ is an example of $\G$-flat ideal. The theorem below motivates the above definition.

\begin{theo}\label{theo:thetheoremainFAnewnew}
Let $X$ be a Banach space,  $Y \in \ACK_\rho$ , $\G \subset Y^*$ be the corresponding $1$-norming set from Defintion~\ref{def_ACK} and $T\in L(X,Y)$ be a $\Gamma$-flat operator with $\norm{T}=1$. Let $0 < \eps \le 1/2$ and let $x_0\in S_X$ be such that $\norm{Tx_0} > 1-\eps$. Then  there exist $u_0\in S_X$ and an operator $S\in S_{L(X,Y)}$ with $\norm{Su_0}=1$ such that
\begin{equation*}
\max \left\{\norm{x_0-u_0}, \norm{T-S}\right\} <  \sqrt{2\eps}\left( 1 + \frac{2}{1-\rho+\sqrt{2\eps}} \right).
\end{equation*}
Moreover, if  $Y\in\ACK$ then the estimate can be improved to
\begin{equation*} \label{eq:YinACK}
\max \left\{\norm{x_0-u_0}, \norm{T-S}\right\} <  \sqrt{2\eps}.
\end{equation*}
Additionally, $S$ can be chosen from $\mathcal I$ whenever $T$ belongs to a $\G$-flat ideal $\mathcal I$. In particular, every  $Y \in \ACK_\rho$ ($\ACK$) has the  A-BPBp.
\end{theo}

Before proving the theorem, we need a preliminary result.
\begin{lem}\label{theo:thetheoremainFAnew}
Under the conditions of Theorem \ref{theo:thetheoremainFAnewnew} above, for every $k \in \left(\eps/(2(1-\eps)), 1\right)$ and for every
\[
\nu >  2k\left( 1 + \frac{2}{1-\rho+2k} \right),
\]
there exist $u_0\in S_X$ and $S\in S_{L(X,Y)}$ satisfying $\norm{Su_0}=1$, $\norm{x_0-u_0}\leq\frac{\eps}{k}$ and $ \norm{T-S} <   \nu$. In the case of $Y \in \ACK$ the same is true for every $\nu > 2k$.

If, moreover, $T$ belongs to a $\G$-flat ideal $\mathcal I$, then $S$ can be chosen from $\mathcal I$ as well.
\end{lem}
\begin{proof}
First, consider the more involved case of  $Y \in \ACK_\rho$. Fix $r>0$ and $0<\eps'<2/3$.  Now, we
can apply Lemma~\ref{ACK-Lem} {with} $Y$,
$\Gamma$, $r$ and
$\eps > 0$. We produce a $w^*$-open set $U_r \subset Y^*$ with $U_r \cap \G \neq\emptyset$, and   points $x_r^*\in S_{X^*}$ and $u_r\in S_X$ with $|x_r^*(u_r)|=1$ such that \eqref{eq:newACKlemma} holds true.

Since $U_r \cap \G \neq\emptyset$, we can apply Definition \ref{def_ACK} to $U =  U_r  \cap \G$ and $\eps'$ and obtain a {non-empty} $V  \subset U$, $y_1^* \in V$, $e \in S_Y$, $F \in L(Y)$ and $V_1 \subset \G$  which satisfy properties (I) -- (VI). In particular, for every $z^* \in V \subset U_r \cap \G$ according to \eqref{eq:newACKlemma}  we have
 \begin{equation} \label{eqT^*v}
\norm{T^*z^* - x_r^*}\leq  r + 2k.
\end{equation}

Define now the linear operator $S\colon X\to Y$ by the formula
\begin{equation}\label{eq:operadorDef-new}
S(x)\coloneqq x_r^*(x)Fe+(1-\widetilde{\eps})(I_Y-F)Tx,
\end{equation}
where the value of $\widetilde{\eps} \in [\eps', 1)$ will be specified below in such a way that $\norm{S}\leq 1$. In order to do this, bearing in mind the fact that $\G$ is $1$-norming,  we can write
\[
\norm{S} = \norm{S^*}  = \sup\left\{\norm{S^*y^*}: y^* \in \Gamma\right\}.
\]
So our first goal is to estimate
\begin{equation} \label{widetilde{T}^*y^*}
\norm{S^*y^*} = \norm{y^*(Fe)x_r^* + (1-\widetilde{\eps}){T}^*(I_{Y^*} - F^*)(y^*)}
\end{equation}
from above for all $y^* \in \Gamma$. For $y^* \in V_1$, the sought estimate $\norm{S^*y^*} \le 1$ follows immediately from the definition of  $V_1$ (see property (IV)). So, it remains to consider the case  $y^* \in \G \setminus V_1$. Thanks to (V), for every $y^* \in \G$, there exists an element  $v^* = \sum_{k=1}^n \lambda_k v_k^*$ with
\begin{equation}\label{eq*F^*-v^*}
\norm{F^*y^* - v^*} < \eps'
\end{equation}
such that $\{v_k^*\}_{k=1}^n \subset  V$,  and $\sum_{k=1}^n |\lambda_k| \le 1$. According to \eqref{eqT^*v} we have  $\norm{T^*v_k^* - x_r^*}\leq r+ 2k$, consequently
\begin{align}  \label{eq*T^*v^*}
\nonumber \norm{v^*(e)x_r^*  - {T}^*v^*}  &\le   \sum_{k=1}^n  |\lambda_k| \norm{ v_k^*(e)x_r^*  - {T}^* v_k^*}  \\
                                          &\stackrel{\mathrm{(VI)}}\le \eps'+ \sum_{k=1}^n  |\lambda_k| \norm{x_r^*  - {T}^* v_k^*} \leq  \eps'+ r+ 2k.
\end{align}
Now, for every  $y^* \in \G \setminus V_1$
\begin{align*}
\norm{S^*y^*} & \le \widetilde{\eps}\, |y^*(Fe)| + (1-\widetilde{\eps}) \norm{y^*(Fe)x_r^* +{T}^*y^* - {T}^*F^*y^*} \\
& \stackrel{\mathrm{(IV)}}\le \widetilde{\eps} \rho +  (1-\widetilde{\eps}) \norm{{T}^*y^*} + (1-\widetilde{\eps}) \norm{(F^*y^*)(e)x_r^*  - {T}^*F^*y^*} \\
&  \stackrel{\eqref{eq*F^*-v^*}}\le  \widetilde{\eps}\rho +   (1-\widetilde{\eps})  + 2\eps' (1-\widetilde{\eps})+ (1-\widetilde{\eps}) \norm{v^*(e)x_r^*  - {T}^*v^*}  \\
&  \stackrel{\eqref{eq*T^*v^*}}\le  \widetilde{\eps}\rho + (1-\widetilde{\eps}) + 2\eps' (1-\widetilde{\eps})+ (1-\widetilde{\eps})(\eps' +  r + 2k)  \\
& \le  \widetilde{\eps}\rho + (1-\widetilde{\eps})(1 + 3\eps' + r+ 2k).
\end{align*}

This means, that if we choose $\widetilde{\eps} = (3 \eps' +r+ 2k)/(1 - \rho + 3\eps' + r+ 2k)$, then we have $\|S\|\leq 1$. In this case,
    \[
        1=|x_r^*(u_r)| \stackrel{\mathrm{(II)}}{=}|y_1^*(x_r^*(u_r)Fe)|\stackrel{\mathrm{(III)}}{=}|y_1^*(S(u_r))|\leq\norm{S(u_r)}\leq 1.
    \]
Therefore, $\norm{S} = 1$ and $S$ attains the norm at the point $u_0\coloneqq u_r \in S_X$ for which by \eqref{eq:newACKlemma} we already had that $\norm{u_0-x_0}\leq\frac{\eps}{k}$.

Now, let us estimate
\begin{align}\label{eq:T-T}
\nonumber \norm{S - T}  &= \norm{S^* - T^*} = \sup_{y^* \in \Gamma}\norm{S^*y^* - T^*y^*} \\
  &\le \sup_{y^* \in \Gamma}\norm{y^*(Fe)x_r^* - {T}^* F^*y^*} + 2 \widetilde{\eps}.
\end{align}

For every $y^* \in \G$ we can proceed the same way as before. Namely,
\begin{align*}
\norm{(F^*y^*)(e)x_r^* - {T}^* F^*y^*}  & \stackrel{\eqref{eq*F^*-v^*}}\le  2\eps' +  \norm{v^*(e)x_r^*  - {T}^*v^*}  \\
&  \stackrel{\eqref{eq*T^*v^*}}\le 3\eps'+  r + 2k.
\end{align*}
Combining this with the inequalities \eqref{eq:T-T} and the value of $\widetilde\eps$ we conclude that
\begin{equation} \label{eqqqqq}
        \norm{T-S}\leq  3\eps'+  r + 2k + 2\frac{3 \eps' +  r + 2k}{1 - \rho + 3\eps' +  r + 2k}.
\end{equation}

Since $r>0$ and  $0<\eps'<2/3$ are arbitrary, for suitable values we will have the desired estimate $ \norm{T-S} <   \nu$.

To finish the proof in the case of  $Y \in \ACK_\rho$ we observe that if $T$ belongs to a $\G$-flat ideal $\mathcal I$ then $S \in \mathcal I$.

\vspace{3mm}
Now the simpler case of  $Y \in \ACK$. In this case $\norm{S^*y^*} \le 1$ for all $y^* \in \G$ thanks to (IV)'. So, $\|S\|\leq 1$ for all values of $\widetilde{\eps} \in [\eps', 1)$  and we can simply take $\widetilde{\eps} = \eps'$. With such a choice of $\widetilde{\eps}$ the estimate \eqref{eqqqqq} changes to $  \norm{T-S}\leq  5\eps'+  r + 2k $, which again for small values of $r$ and $\eps'$ gives us $ \norm{T-S} <   \nu$ for the $\nu$ which corresponds to this case.
\end{proof}

\begin{proof}[Proof of Theorem \ref{theo:thetheoremainFAnewnew}]
First, select $\eps_0 \in (0, \eps)$ in such a way that the inequality $\norm{Tx_0} > 1-\eps_0$ is still valid. Now we  apply Lemma \ref{theo:thetheoremainFAnew} with $\eps_0$ instead of $\eps$ and substitute   $k = \sqrt{\eps_0/2}$. In the case of  $Y \in \ACK_\rho$ we take $\nu \in \Bigl(\sqrt{2\eps_0}\left( 1 + \frac{2}{1-\rho+\sqrt{2\eps_0}}\right), \sqrt{2\eps}\left( 1 + \frac{2}{1-\rho+\sqrt{2\eps}}\right)\Bigr)$, and  in the case of  $Y \in \ACK$  we take $\nu \in (\sqrt{2\eps_0},  \sqrt{2\eps})$.
\end{proof}

\begin{rem}
The statements of Lemma \ref{theo:thetheoremainFAnew} and Theorem \ref{theo:thetheoremainFAnewnew} remain correct if in the definition of $\ACK_\rho$ and $\ACK$ the property (IV) is substituted by the following weaker one, in which $V_1$ is larger than in the original definition:

Denote $V_1 = \{y^* \in \G :  |y^*(Fe)| + (1 - \eps')\norm{(I_{Y^*} - F^*)(y^*)} \le 1 \}$. Then  $|v^*(Fe)| \le  \rho$ for every $v^* \in \G \setminus V_1$.

Also, a look at the proof of Lemma \ref{ACK-Lem} shows that the condition of  $T$ being $\G$-flat can be weaken in the following way: for every $y \in B_Y$ and every  $\delta > 0$ if the $w^*$-slice $S(\G, x, \delta) \coloneqq  \{y^* \in \G: \re y^*(y) > 1 - \delta\}$ is not empty, then for every  $\eps > 0$ there exists  a non-empty relatively $w^*$-open subset $V \subset S(\G, x, \delta)$  such that $\diam (T^*(V)) < \eps$.

There are two reasons why we have selected the more restrictive variants. Firstly, with the restrictive definition of (IV) we are able to prove a nice stability result (Theorem \ref{theo:InjectiveTensorProduct} below), and secondly, all the examples with ``relaxed'' versions of (IV) and of  $\G$-flatness that we have in hand, satisfy the restrictive variant of (IV) and of  $\G$-flatness.
\end{rem}


\section{Banach spaces with $\ACK$ structure}

The aim of this section is presenting those \emph{natural} examples of Banach spaces having $\ACK$ structure as well as showing the stability of the $\ACK$ structure under some operations, such us $\ell_\infty$-sums or injective tensor products.


\medskip

First of all, let us introduce the first natural class of Banach spaces with $\ACK$ structure. As commented above, Definition~\ref{def_ACK}, comes from an analysis of the proofs in \cite{CasGuiKad}. We shall show next that, indeed, every uniform algebra $\algebra$ has simple $\ACK$ structure. The key tool is Lemma~\ref{Ur-alg}, that was proved in~\cite[Lemma 2.5 and Lemma 2.7]{CasGuiKad}, and is about the existence of peak functions  $f \in S_\algebra$  whose range is contained in the Stolz's region
   \[
      \st_\eps = \{z\in \mathbb{C}: |z| + (1-\eps)|1-z|\leq 1\}.
   \]

For a topological space $(T,\tau)$, we denote by $C_{b}(T)$ the space of bounded continuous functions $f\colon T \to \K$ equipped with the sup-norm.

\begin{defi} \label{algebroid subspace}
   Let $(T,\tau)$ be a topological space. A subalgebra $\algebra \subset C_{b}(T)$ is said to be an \emph{$\ACK$-subalgebra}, if for every {non-empty} open set $W \subset T$  and $0<\eps<1$, there exist $f\in\algebra$ and $t_0 \in W$ such that $f(t_0)=\norm{f}_{\infty}=1$, $|f(t)|<\eps$ for every $t\in T \setminus W$ and $f(T)\subset \st_\eps$.
\end{defi}

\begin{lem}\label{Ur-alg}
 Let $ \algebra\subset C(K)$ be a uniform algebra. Then there exists a topological space $\G_\algebra$ such that $\algebra$ is  isometric to an $\ACK$-subalgebra of $C_b(\G_\algebra)$. In the case of  $K$ being the space of multiplicative functionals on $\algebra$ the corresponding $\G_\algebra$ can be selected as a topological subspace of $K$.
\end{lem}

We will use the following elementary property of $\st_\eps$.

\begin{lem} \label{stolzpowern}
   If $z$ belongs to the Stolz region $\st_\eps$, then $z^n \in\st_\varepsilon$.
\end{lem}

\begin{proof}
   For every $z\in\st_\eps$ it holds
   \begin{align*}
      |z^n| + (1-\varepsilon)|1 - z^n| &= |z^n| + (1-\varepsilon)|1 - z||1 + z + \ldots + z^{n-1}|\\
                                       &\leq |z|^n + (1 - |z|)|1 + z + \ldots + z^{n-1}|\\
									   &\leq |z|^n + (1 - |z|)(1 + |z| + \ldots + |z|^{n-1}) \\
									   &= |z|^n + (1 - |z|^n) = 1,
   \end{align*}
   which finishes the proof.
\end{proof}

The following simple lemma gives an essential property that turns uniform algebras into Banach spaces with simple $\ACK$ structure.

\begin{lem}\label{lem-e-f}
   Let $\algebra\subset C_{b}(\G_\algebra)$ be an $\ACK$-subalgebra. Then, for every non-empty open set $W\subset \G_\algebra$ and $0<\eps<1$, there exist a non-emtpy subset $W_0 \subset W$, functions $f$, $e\in\algebra$, and $t_0 \in W_0$ such that $f(t_0)=\norm{f} = 1$, $e(t_0)=\norm{e} = 1$, $|f(t)|< \eps$ for every $t\in \G_\algebra \setminus W_0$, $|1-e(t)|<\eps$ for every $t\in W_0$ and  $f(\G_\algebra)\subset \st_\eps$.
\end{lem}

\begin{proof}
   By using Definition \ref{algebroid subspace} for the open set $W \subset \G_\algebra$ and $\eps$, we get a function $e\in\algebra$ and $t_0\in W$ such that $e(t_0)=\norm{e}=1$, $|e(t)|<\eps$ for every $t\in \G_\algebra \setminus W$ and $e(\G_\algebra)\subset \st_\eps$. Let $W_0\coloneqq \{t\in W: |1-e(t)| < \eps\}$. Define the function $f_n\colon \G_\algebra\to\K$ by $f_n(t)\coloneqq (e(t))^n$ whose range, by Lemma \ref{stolzpowern}, is contained in $\st_\eps$. From the very definition of $W_0$ and the fact that $e(\G_\algebra)\subset\st_\eps$, we deduce that $|e(t)|\leq 1-\eps(1-\eps)<1$ for every $t\in\G_\algebra\setminus W_0$. Thus, taking a suitable $n_0\in\mathbb{N}$, we can assume that $|f_{n_0}(t)|=|e(t)|^{n_0}<\eps$ on $\G_\algebra\setminus W_0$. Therefore, $f\coloneqq f_{n_0}\in\algebra$ gives the conclusions of the lemma.
\end{proof}

\begin{theo}\label{theo:unalg->ACK}
   Let $\algebra \subset C_{b}(\G_\algebra)$ be an $\ACK$-subalgebra, and let $X$ be a subspace $\algebra\subset X\subset  C_{b}(\G_\algebra)$ that has the following property: $fx \in X$ for every $x \in X$ and $f \in\algebra$. Then $X \in \ACK$ with the corresponding $1$-norming subset of $B_{X^*}$ being $\G=\{\delta_t: t\in \G_\algebra\}$.
\end{theo}

\begin{proof}
  Fix $\eps>0$ and a non-emtpy relatively $w^*$-open subset $U = \{\delta_t: t \in W \subset \G_\algebra\}\subset\G$. Observe that $W \subset \G_\algebra$ is open. Now, by applying  Lemma~\ref{lem-e-f} to $W$ with $\eps$ we obtain the corresponding $W_0\subset \G_\algebra$, $t_0\in W_0$, $f$, $e_\algebra\in\algebra$. Let us define  $V\subset U$, $x_1^* \in V$, $e \in S_{X}$ and $F \in L(X)$ as follows:
\[
V\coloneqq \{\delta_t: t\in W_0\},\quad
x_1^*\coloneqq {\delta_{t_0}},\quad
e\coloneqq e_\algebra,\quad
Fx\coloneqq f x,\,\text{for}\,x\in X.
\]
Then, $F^*x^* = f(t) x^*$  for every $x^* = \delta_t \in \G$. We shall show that properties (I) -- (VI) are satisfied. First, $\norm{F} \leq1$ and $\norm{Fe} = e(t_0) f(t_0) = 1$, which proves (I). Property (II) is straightforward from $x_1^*(Fe) = x^*_1(f e) = e(t_0) f(t_0) = 1$. From $(F^*x_1^*)(x) = x(t_0) f(t_0) = x(t_0) = x^*_1({x})$ we deduce that $F^*x_1^* = x_1^*$, which is (III). To show (IV)', take $x^*=\delta_t\in\G$ and estimate
\begin{align*}
\norm{F^*{x^*}} &+ (1 - \eps)\norm{(I_{{X^*}} - F^*)({x^*})} \\
              &\leq |f(t)|+(1-\eps)|1-f(t)|\le 1.
\end{align*}
Let us show now (V). Take $x^* = \delta_t \in \G$. In case $t$ belongs to $\G_\algebra\setminus W_0$, then $\norm{F^*x^*} = |f(t)| <\eps$. Otherwise, $t\in W_0$ (that is, $x^*\in V$), using that $F^*x^* =  f(t) x^*$ and that $f\in S_X$, we deduce that  $f(t) x^* \in \aconv\{0, V\}$. Hence, in both cases
\[
\dist(F^*x^*, \aconv\{0, V\}) < \eps.
\]
Finally, for every $v^* \in V$ we have that $v^*(e)=e(t)$ for some $t\in W_0$. So,
\[
|v^*(e) - 1| = |e(t)-1|\le  \eps,
\]
which shows (VI) and finishes the proof.
\end{proof}

From Lemma \ref{Ur-alg} and Theorem \ref{theo:unalg->ACK} taking $X = \algebra$ we obtain the promised example.
\begin{cor}\label{cor:unalg->ACK}
Every uniform algebra $\algebra$ has simple $\ACK$ structure.
\end{cor}

Theorem \ref{theo:unalg->ACK} gives more examples of spaces with simple $\ACK$ structure. For instance, let $\T$ be the unit disk in $\C$, $ A(\T) \subset C(\T)$ be the disc-algebra, i.e., $ A(\T) $ is the closure in   $C(\T)$ of the set $\{\sum_{k=0}^m a_k z^k : a_k \in \C, m \in \N\}$ of all polynomials.  For a given $n \in \N$ denote $A_n(\T)$ the closure in   $C(\T)$ of the set $\{\sum_{k = -n}^m a_k z^k : a_k \in \C, m \in \N\}$. Then $ A(\T) $ and ${X} = A_n(\T)$ satisfy all the conditions of  Theorem \ref{theo:unalg->ACK}, so $A_n(\T) \in \ACK$, but $A_n(\T)$ is not an algebra.
Another example: let  $c_0 \subset {X} \subset \ell_\infty$. Then ${X} \in \ACK$.

The first example is of illustrative character, because the space $A_n(\T)$ is isometric to the algebra $A(\T)$. In contrast, the second example gives a big variety of mutually non-isomorphic spaces with $\ACK$ structure. Observe that the simple $\ACK$ structure of those {$X$ such} that   $c_0 \subset {X} \subset \ell_\infty$ can be also deduced from Theorem \ref{theo:beta->ACK} below.

\begin{rem}
In general, it is not clear whether for a given  $T \in \Fl(X,Y)$ the formula \eqref{eq:operadorDef-new} gives a $\G$-flat operator  $S$. But, under the conditions of Theorem \ref{theo:unalg->ACK},
we have an additional property {$F^*x^* = f(t) x^*$}. Combining this property with  (iv) of Theorem \ref{theo:op-fragm-comp}, we get $S \in \Fl(X,Y)$. In particular, in the case of uniform algebras the Bishop--Phelps--Bollob\'as type approximation of $\G$-flat operators can be made by operators that are $\G$-flat as well.
\end{rem}

Now we show that Banach spaces with Lindenstrauss' property  $\beta$ (see for instance~\cite{Lind1963}) have $\ACK$ structure.

\begin{defi} \label{def:beta} A Banach space $X$ is said to have the property $\beta$ if there exist two sets $\{x_\alpha: \alpha\in \Lambda\}\subset S_X$, $\{x^*_\alpha: \alpha\in  \Lambda\}\subset S_{X^*}$ and $\rho\in[0,1)$ such that the following conditions hold:
\begin{enumerate}
\item $x^*_\alpha(x_\alpha) = 1$;
\item $|x^*_\alpha(x_\gamma)| \leq \rho < 1$ if $\alpha \neq \gamma$; and
\item $\norm{x} = \sup\{|x_\alpha^*(x)|: \alpha \in \Lambda\}$, for all $x\in X$.
\end{enumerate}
\end{defi}

\begin{theo}\label{theo:beta->ACK}
Let $X$ have the property $\beta$. Then $X \in \ACK_\rho$ with the same value of $\rho$ as in Definition \ref{def:beta} and with $\Gamma=\{x_\alpha^*: \alpha \in \Lambda\}$ from that definition. Moreover, if $X$ has property $\beta$ with $\rho = 0$, then $X \in \ACK$.
\end{theo}

\begin{proof}
Since $X$ has property $\beta$, the set $\G=\{x_\alpha^*: \alpha \in \Lambda\}$ is a $1$-norming subset of $B_{X^*}$. Observe that property $\beta$ implies that $(\G,w^*)$ is a discrete topological space. Fix $\eps>0$ and a non-empty relatively $w^*$-open subset $U\subset\G$. Take $x_{\alpha_0}^*\in U$. Let us define the corresponding $V$, $x_1^* \in V$, $e \in S_{X}$, and $F \in L(X)$ as follows:
\[
V\coloneqq \{x^*_{\alpha_0}\}\subset U,\quad
x_1^*\coloneqq x_{\alpha_0}^*,\quad
e\coloneqq x_{\alpha_0},\quad
F(x)\coloneqq x_{\alpha_0}^*(x) x_{\alpha_0}.
\]
It is clear that $F^*x^* = x^*(x_{\alpha_0}) x_{\alpha_0}^*$ for every $x^*\in X^*$. We shall show that properties (I) -- (VI) of Definition~\ref{def_ACK} hold true. Properties (I) -- (III) are routine. To show (IV) observe first that
\begin{align*}
 \norm{F^*x^*_{\alpha_0}} + (1 - \eps)\norm{(I_{X^*} - F^*)(x^*_{\alpha_0})} = \norm{x^*_{\alpha_0}(x_{\alpha_0})x^*_{\alpha_0}}  = 1,
\end{align*}
that is, $x^*_{\alpha_0} \in V_1$. Consequently, whenever $v^* =  x^*_{\alpha} \in \G \setminus V_1$, then $\alpha \neq \alpha_0$ and thus $|v^*(Fe)| = |x_\alpha^*(x_{\alpha_0})| \le  \rho$.

In case that $\rho = 0$, we have that $F^*x^*_{\alpha} = 0$ for every  $\alpha \neq \alpha_0$, so
\begin{align*}
 \norm{F^*x^*_{\alpha}} + (1 - \eps)\norm{(I_{X^*} - F^*)x^*_{\alpha}} = (1 - \eps) \norm{x^*_{\alpha}}  < 1,
\end{align*}
i.e., $V_1 = \G$.

Property (V) is a consequence of the fact that $F^*x^* \in \aconv\{0, V\}$ for every $x^* =  x^*_{\alpha} \in \G$, because $F^*x^*= x^*_{\alpha}(x_{\alpha_0})x^*_{\alpha_0}$. Finally, property (VI) and in turn our conclusions are consequence of the fact that the unique $v^* \in V$ is $v^* = x^*_{\alpha_0}$, so $|v^*(e) - 1| = 0 \le  \eps$.
\end{proof}

\begin{cor}[{\cite[Theorem 2.2]{aco-aro-gar-mae}}] \label{coroll:beta->ACK}
   Let $Y$ have property $\beta$. Then, for every Banach space $X$, the pair $(X, Y)$ has the Bishop--Phelps--Bollob\'{a}s property for operators.
\end{cor}

\begin{proof}
In the proof of Theorem \ref{theo:beta->ACK}, $(\G, w^*)$ is a discrete topological space. Therefore every operator $T \in L(X,Y)$ is $\G$-flat (Example~\ref{example-Discrete-are-G-Flat} after Definition \ref{def_gam-flat-op}). Now the application of  Theorem \ref{theo:thetheoremainFAnewnew} completes the proof.
\end{proof}


Now we show the stability of the $\ACK$ structure with respect to the operations of $\ell_\infty$-sum and injective tensor product of two spaces (Theorem~\ref{theo:Y_1Y_2} and Theorem~\ref{theo:InjectiveTensorProduct})

\begin{theo}\label{theo:Y_1Y_2}
Let $X$, $Y$ be Banach spaces having $\ACK$ structure with parameters $\rho_X$ and $\rho_Y$ respectively. Then $Z\coloneqq X\bigoplus_\infty Y \in \ACK_{\rho}$ with  $\rho= \max\{\rho_X, \rho_Y\}$. Moreover,  $Z \in \ACK$ whenever $X$, $Y \in \ACK$.
\end{theo}

\begin{proof}
 Observe that both $X$ and $Y$ have $\ACK$ structure with parameter $\rho$. Let $\Gamma_X \subset B_{X^*}$ and $\Gamma_Y \subset B_{Y^*}$  be  the corresponding $1$-norming subsets in Definition \ref{def_ACK}.  Then, the set
\[
\Gamma\coloneqq \{(x^*,0): x^*\in  \Gamma_X\} \cup \{ (0, y^*): y^*\in  \Gamma_Y\}
\]
is a $1$-norming subset of $B_{Z^*}$.  Take a non-empty relatively $w^*$-open subset $U\subset\G$. Then, there exist relatively $w^*$-open subsets $U_X\subset\G_X$ and $U_Y\subset\G_Y$ that are not both empty and such that $(U_X\times\{0\})\cup (\{0\}\times U_Y)\subset U$. Without loss of generality we may assume that $U_X \neq \emptyset$.

Fix $\eps>0$. By using Definition~\ref{def_ACK} for $X$, $\eps $, and $U_X$ we obtain  a non-empty subset $V_X\subset U_X$,  $x_1^* \in V_X$, $e_X \in S_{X}$, $F_X \in L(X)$  with the properties (I) -- (VI).  Thus, we can define the  corresponding $V\subset U$, $z_1^* \in V$, $e \in S_{Z}$ and $F \in L(Z)$ as follows:
\[
V\coloneqq \{(x^*,0): x^*\in V_X\}\subset U,\quad z_1^*\coloneqq (x_1^*,0),\quad e\coloneqq (e_X, 0),
\]
and for $(x,y)\in Z$,
\[
F(x,y)\coloneqq(F_X(x),0).
\]

Let us check the required properties. It is clear that $\norm{F} =1$ and that $\norm{Fe} = \norm{{F_X(e_X)}} = 1$, which shows (I). (II) follows easily; $z^*(Fe) = x_1^*(F_Xe_X)= 1$.  Due to the fact that  $(F_X x_1^*,0) = (x_1^*, 0)$, we deduce that $F^*z_1^* = z_1^*$, showing that (III) holds. Now, for every $z^*=(x^*,0) \in V$ with $x^* \in V_{X,1}$ we have
\begin{align*}
\norm{F^*z^*} &+ (1 - \eps)\norm{(I_{Z^*} - F^*)(z^*)} \\
& = \norm{F_X^*x^*} + (1 - \eps)\norm{(I_{X^*} - F_X^*)(x^*)}\\
&\le 1,
\end{align*}
which can be easily deduced {from} $F^*z^*=(F_X^*x^*,0)$. Consequently, for every  $x^* \in V_{X,1}$ we have $z^*=(x^*,0) \in V_{1}$. (Observe that in the case of \emph{simple ACK structure} we have already proved (IV)'). Let $v^*  \in \G \setminus V_1$.  Then, either $v^*=(0, y^*)$, or $v^*=(x^*, 0)$ with $x^* \in \G_X \setminus V_{X,1}$.   On the one hand, when $v^*=(0, y^*)$, we have $|v^*(Fe)| = 0 \leq \rho$. On the other hand, whenever  $v^*=(x^*, 0)$ with $x^* \in \G_X \setminus V_{X,1}$, then $|v^*(Fe)|= |x^*(F_Xe_X)| \le  \rho$, which proves (IV). Now, let ${z^*} \in \G$. Whenever $z^* = (0, y^*)$ we have $F^*z^* = 0$. Otherwise, $z^* = (x^*,0)$ and we have $\dist({F_X^* x^*}, \aconv\{0, {V_X}\}) < \eps$. Thus, in both cases
\[\dist({F^*z^*}, \aconv\{0, V\}) < \eps.\]
Finally, for every $v^* = (x^*, 0) \in V$ we have $|v^*(e) - 1| = |x^*(e_X) - 1| \le  \eps$, which proves (VI) and concludes our proof.
\end{proof}

Recall, that given two normed spaces $X$ and $Y$, one can define their injective tensor product $X\itp Y$, as the completion of $(X\otimes Y,\norm{\cdot}_\varepsilon)$, where
    \[
        \norm{z}_\varepsilon\coloneqq \sup\{|\langle x^*\otimes y^*,z\rangle|\colon x^*\in B_{X^*},\, y^*\in B_{Y^*}\},
    \]
for every $z\in X\otimes Y$ and $\langle x^*\otimes y^*,x\otimes y\rangle\coloneqq  x^*(x)\,y^*(y)$, for every $x\otimes y\in X\otimes Y$ and for every $x^*\in X^*$ and $y^*\in Y^*$.

An important example of such a product is the Banach space $C(K) \itp Y$, which can be naturally identified with $C(K,Y)$, that is, the Banach space of continuous $(Y,\norm{\cdot})$-valued functions defined on $K$, endowed with the supremum norm $\norm{f}=\sup\{\norm{f(t)}: t\in K\}$.

Note that it follows from the definition of the injective norm that if $X_0\subset B_{X^*}$ and $Y_0\subset B_{Y^*}$ are $1$-norming, then for every $z\in X\itp Y$ the following equality holds:
    \[
        \norm{z}_\eps=\sup\{|\langle x^*\otimes y^*,z\rangle|: x^*\in X_0,\, y^*\in Y_0\}.
    \]
Recall also that $\norm{x^*\otimes y^*}_{(X\itp Y)^*} = \norm{x^*} \cdot \norm{y^*}$ for every $x^*\in X^*$ and $y^*\in Y^*$.

This is all the information about tensor products that will be used in Theorem \ref{theo:InjectiveTensorProduct} below.  We refer to Ryan's book \cite{Ryan} for tensor products theory in general and the above definitions and statements in particular.

\begin{theo}\label{theo:InjectiveTensorProduct}
Let $X$ and $Y$ be Banach spaces both of which have $\ACK$ (resp. $\ACK_\rho$) structure. Then,  $X\itp Y$ has  $\ACK$ (resp. $\ACK_\rho$) structure.
\end{theo}

\begin{proof}
Since $X$ and $Y$ have  $\ACK$ (resp. $\ACK_\rho$) structure, there exist $1$-norming sets $\Gamma_X\subset S_{X^*}$ and $\Gamma_Y\subset S_{Y^*}$ satisfying Definition~\ref{def_ACK}. Define the map $\phi\colon (B_{X^*},w^*)\times (B_{Y^*},w^*)\to (B_{(X\itp Y)^*},w^*)$ by $\phi(x^*,y^*)=x^*\otimes y^*$,  for every $x^*\in B_{X^*}$ and for every $y^*\in B_{Y^*}$.

    First, we shall show that the map $\phi$ is continuous. Let $\{(x^*_\alpha,y^*_\alpha)\}_{\alpha\in\Lambda}$ be a convergent net to $(x^*,y^*)\in B_{X^*}\times B_{Y^*}$. Then, for every $x\otimes y\in X\otimes Y$, we can estimate
    \begin{align*}
        |\langle \phi(x^*_\alpha,y^*_\alpha)&-\phi(x^*,y^*),x\otimes y\rangle|=|x^*_\alpha(x)y^*_\alpha(y)-x^*(x)y^*(y)|\\
                                            &\leq |(x^*_\alpha(x)-x^*(x))y^*_\alpha(y)|+|x^*(x)(y^*_\alpha(y)-y^*(y))|\\
                                            &\leq |x_\alpha^*(x)-x^*(x)|\norm{y_\alpha^*}\norm{y}+\norm{x^*(x)}|y^*_\alpha(y)-y^*(y)|\\
                                            &\leq |x_\alpha^*(x)-x^*(x)|\norm{y}+\norm{x}|y^*_\alpha(y)-y^*(y)|,
    \end{align*}
    which tends to zero. This argument extends easily to every element in $X\otimes Y$ and, in turn, to every $z\in X\itp Y$ (due to the boundedness of the range of the map $\phi$).

   The $1$-norming set $\G$ that we need for our theorem can be introduced as follows:
    \[
        \Gamma\coloneqq \{x^*\otimes y^*\colon x^*\in\Gamma_X,\, y^*\in \Gamma_Y\}=\phi(\Gamma_X\times\Gamma_Y).
    \]

    Let $\eps>0$ and $U$ be a {non-empty} relatively $w^*$-open subset of $\Gamma$. Let $x^*_0\in \Gamma_X$ and $y^*_0\in \Gamma_Y$ be such that $\phi(x^*_0,y^*_0)\in U$. The continuity of $\phi$ ensures that there exist non-empty relatively $w^*$-open  subsets $W_X\subset \Gamma_X$,  $W_Y\subset \Gamma_Y$ such that $x^*_0\in W_X$, $y^*_0\in W_Y$ and $\phi(W_X\times W_Y)\subset U$.

    We can apply Definition~\ref{def_ACK} to $X$ and $Y$, to the former with $\eps/2$ and $W_X$ and to the latter with $\eps/2$ and $W_Y$, to find two {non-empty} sets $V_X\subset W_X$ and $V_Y\subset W_Y$, two functionals $x_1^*\in V_X$ and $y_1^*\in V_Y$, two points $e_X\in S_X$ and $e_Y\in S_Y$ and finally, two operators $F_X\in L(X)$ and $F_Y\in L(Y)$, satisfying respectively the properties (I) -- (VI), or with their corresponding modifications for the the simple $\ACK$ structure. Denote also by $V_{X,1}$ and  $V_{Y,1}$ the corresponding variants for $X$ and $Y$ of the set $V_1$ from property (IV) of  Definition~\ref{def_ACK}.

    Now, define the non-emtpy set $V\subset U$ and corresponding $z_1^*\in V$, $e\in S_{X\otimes_\varepsilon Y}$, $F\in L(X\itp Y)$ as follows: $V\coloneqq \phi(V_X\times V_Y)\subset U$, $z_1^*\coloneqq \phi(x^*_1, y_1^*)=x^*_1\otimes y^*_1$, $e\coloneqq e_X\otimes e_Y$, and $F(x\otimes y)\coloneqq F_X(x)\otimes{F_Y}(y)$ for every $x\otimes y \in X\otimes Y$. It remains to check the properties (I) -- (VI). First, observe that $F^*(x^*\otimes y^*)=F_X^*x^*\otimes F_Y^*y^*$ for every $x^*\in X^*$ and $y^*\in Y^*$.

\item[(I)]
    Let $z$ belong to $B_{X\itp Y}$, then
    \begin{align*}
        \norm{Fz}_\eps &=\sup_{x^*\in\Gamma_X}\sup_{y^*\in\Gamma_Y}|\langle x^*\otimes y^*,Fz\rangle|=\sup_{x^*\in\Gamma_X}\sup_{y^*\in\Gamma_Y}|\langle F^*(x^*\otimes y^*),z\rangle|\\
                    &=\sup_{x^*\in\Gamma_X}\sup_{y^*\in\Gamma_Y}|\langle F_X^*x^*\otimes F_Y^*y^*,z\rangle|\leq \sup_{x^*\in\Gamma_X}\sup_{y^*\in\Gamma_Y}\norm{F_X^*x^*}\norm{F^*_Yy^*}\\
                    &\leq \norm{F_X^*}\norm{F_Y^*}\leq 1,
    \end{align*}
    which implies that $\norm{F}= 1$, since
    \[
        \norm{Fe} = \norm{F_Xe_X\otimes F_Ye_Y} = \norm{F_Xe_X}\norm{F_Ye_Y}=1.
    \]

\item[(II)] $z_1^*(Fe) = (x_1^*\otimes y_1^*)(F_Xe_X\otimes F_Ye_Y)=x^*_1(F_Xe_X)y_1^*(F_Ye_Y)=1 $.

\item[(III)]
    $F^*z_1^* = z_1^*$, since for every $x\otimes y\in X\otimes Y$ we have
    \[
        (F^*z_1^*)(x\otimes y)=(x_1^*\otimes y_1^*)(F_Xx\otimes F_Y y)=(F_X^*x_1^*)(x)(F_Y^*y_1^*)(y),
    \]
    which, in turn, implies that $(F^*z_1^*)(x\otimes y)=x_1^*(x)y_1^*(y)=z^*_1(x\otimes y)$.

\item[(IV)] For $(x^*,y^*)\in\Gamma_X\times\Gamma_Y$, denote $z^* = x^* \otimes y^*$.  Firstly, let us show that for every $x^*\in V_{X,1}$ and  $y^*\in V_{Y,1}$ the functional $z^*$ belongs to $V_1$, i.e., that
\begin{equation*} \label{eqV_1-tenz}
 \norm{F^*z^*} + (1 - \eps)\norm{(I_{(X\itp Y)^*} - F^*)(z^*)} \le 1.
\end{equation*}
First of all, observe that
\begin{align*}
\|x^* &\otimes y^* - F_X^*x^*\otimes F_Y^*y^*\|=\\
&=\norm{x^* \otimes (y^* -  F_Y^*y^*) -  (x^*  - F_X^*x^*)\otimes F_Y^*y^*}\\
&\leq \norm{ y^* -  F_Y^*y^*} +  \norm{F_Y^*y^*}\norm{(x^*  - F_X^*x^*) }.
\end{align*}
Therefore,
    \begin{align*}
 &\norm{F_X^*x^*} \, \norm{F_Y^*y^*} + (1 - \eps)\norm{x^* \otimes y^* - F_X^*x^*\otimes F_Y^*y^*} \\
 &=\norm{F_Y^*y^*} \bigl(  \norm{F_X^*x^*} + (1 - \eps)\norm{(x^*  - F_X^*x^*) }\bigr) + (1 - \eps)\norm{ y^* -  F_Y^*y^*}  \\
 &\le \norm{F_Y^*y^*}  + (1 - \eps)\norm{ y^* -  F_Y^*y^*}  \le 1.
     \end{align*}
This implies that for every  $z^* = x^* \otimes y^* \in \G \setminus V_1$
 we have two possibilities:  either $x^*\notin V_{X,1}$  or $y^* \notin V_{Y,1}$. By symmetry, it is sufficient to consider $x^*\notin V_{X,1}$. In this case $|x^*(F_Xe_X)|\le\rho$, so
 $$
 |z^*(Fe)|=|{x}^*(F_Xe_X)| \, |{y}^*(F_Ye_Y)| \le |x^*(F_Xe_X)|\le\rho.
 $$
\item[(V)]
    We shall show that $\dist(F^*z^*, \aconv\{0, V\}) < \eps$ for every $z^*= x^*\otimes y^* \in \G$. Due to the facts that $\dist( F_X^*x^*, \aconv\{0, V_X\}) < \eps/2 $ and that $\dist( F_Y^*y^*, \aconv\{0, V_Y\}) < \eps/2 $, there exist $v_X^*\in \aconv\{0, V_X\}$ and $v_Y^*\in \aconv\{0, V_Y\}$ such that  $\norm{F_X^*x^*-v_X^*}< \eps/2$ and $\norm{F_Y^*y^*-v_Y^*}< \eps/2$. Then $v^* \coloneqq v_X^*\otimes v_Y^*$ belongs to $ \aconv\{0, V\}$ and
    \begin{align*}
        \norm{F^*z^* - v^*}&\le \norm{(F_X^*x^*-v_X^*)\otimes F_Y^*y^*}+\norm{v_X^*\otimes (F_Y^*y^*-v_Y^*)} \\
               &\le \norm{F_X^*x^*-v_X^*}\norm{F_Y^*y^*}+\norm{v_X^*}\norm{F_Y^*y^*-v_Y^*}\le \eps.
    \end{align*}
 \item[(VI)]
    For every $v^* = x^*\otimes y^* \in V$ we get
    \begin{align*}
        |v^*(e) - 1| &= |x^*(e_X)y^*(e_Y)-1| \le |x^*(e_X)y^*(e_Y)-y^*(e_Y)|\\
                     &+|y^*(e_Y)-1|\leq \frac{\eps}{2}|y^*(e_Y)|+\frac{\eps}{2}\le \eps.
    \end{align*}
This finishes the proof.
\end{proof}


\subsection{Sup-normed spaces of vector-valued functions}

 As we mentioned in the introduction, Acosta,   Becerra Guerrero,  Garc\'{\i}a,  Kim, and Maestre considered  A-BPBp in spaces of continuous vector-valued functions. Let us recall their result explicitly. Here, as usual, $\sigma(Z, \Delta)$ denotes the weakest topology on $Z$ in which all elements of $\Delta \subset Z^*$ are continuous.

\begin{theo}[{\cite[Theorem 3.1]{aco-guer-gar-kim-mae}}] \label{theo:C(K,Y)}
Let $X, Z$  be  Banach spaces, $K$ be a compact Hausdorff topological space. Let $Z$ satisfy property $\beta$ for the subset of functionals $\Delta = \{z^*_\alpha: \alpha \in \Delta\}$. Let $\tau \supseteq \sigma(Z, \Delta)$ be a linear topology on $Z$ dominated by the norm topology. Then for every closed operator ideal $\mathcal I$ contained in the ideal of Asplund operators, we have that $(X, C(K,(Z,\tau)))$ has the Bishop--Phelps--Bollob\'as property for operators from $\mathcal I$.
\end{theo}

The next proposition together with Theorem~\ref{theo:thetheoremainFAnewnew} generalize Theorem \ref{theo:C(K,Y)} for the case of {$Z$} endowed with its strong topology.

\begin{pro}\label{cor:C(K,Y)}
Let $K$ be a compact Hausdorff topological space. Then,
\[
(Y \in \ACK_\rho) \Rightarrow (C(K,Y)  \in \ACK_\rho);
\]
\[
(Y \in \ACK) \Rightarrow (C(K,Y)  \in \ACK).
\]
\end{pro}

\begin{proof}
Bearing in mind {Corollary~\ref{cor:unalg->ACK}} and Theorem~\ref{theo:InjectiveTensorProduct}, the fact that the space $C(K) \itp Y$ is isometric to $C(K,Y)$ concludes the proof.
\end{proof}

Our aim now is showing a generalization of Theorem \ref{theo:C(K,Y)} in the spirit of the $\ACK$ structure, that covers all topologies $\tau$ from that theorem. In order to do this we need some terminology.

For a topological space $T$ and a Banach space $Z$ denote by $C_{\bof}(T, Z)$ the space of all bounded openly fragmented (see Definiton~\ref{def_op-fragm-func}) functions $f\colon T \to Z$ equipped with the $\sup$-norm. For a topology $\tau$ on $Z$ denote by $C_{b}(T, (Z, \tau))$ the space of bounded $\tau$-continuous functions $f\colon T \to Z$ equipped with the sup-norm.

\begin{defi}
Let $Z \in \ACK_\rho$ and let $\G \subset B_{Z^*}$ be the corresponding $1$-norming set.   A  linear topology $\tau$  on $Z$ is said to be \emph{$\G$-acceptable}, if it is dominated by the norm topology and dominates $\sigma(Z, \G)$.
\end{defi}

The following result simultaneously generalizes our Theorem \ref{theo:unalg->ACK} and Theorem \ref{theo:C(K,Y)}. We state the result in the most general settings, which makes the statement bulky. Some ``elegant'' partial cases will be given as corollaries.

\begin{theo} \label{theo:unalg->ACKvector-valued}
   Let $\algebra \subset C_{b}(\G_\algebra)$ be an $\ACK$-subalgebra. Let $Z$ be a Banach space and $\mathcal{O}\subset L(Z)$ such that $Z \in \mathcal{O}$-$\ACK_\rho$ ($Z \in \mathcal{O}$-$\ACK$) with $\G_Z \subset B_{Z^*}$ being the corresponding $1$-norming set. Finally, let $\tau$ be a $\G_Z$-acceptable topology on $Z$. Let $X  \subset  C_{b}(\G_\algebra, (Z, \tau))$ be a Banach space satisfying the following properties:
   \begin{enumerate}
      \item For every $x \in X$ and $f \in\algebra$ the function $fx$ belongs to $X$.
      \item $X$ contains all functions of the form $f\otimes z$,  $f \in\algebra$, $z \in Z$.
      \item $F \circ x \in X$ for every $x \in X$ and $F \in \mathcal{O}$.
      \item For every finite collection $\{x_k\}_{k=1}^n \subset X$ the corresponding function of two variables $\varphi \colon \G_\algebra \times (\G_Z, w^*) \to \K^n$, defined by $\varphi (t, z^*) = (z^*(x_k(t)))_{k=1}^n $, is quasi-continuous.
   \end{enumerate}
   Then $X \in \ACK_\rho$ ($X \in \ACK$, respectively) with the corresponding $1$-norming subset of $B_{X^*}$ being $\Gamma=\{\delta_t \otimes z^*: t\in \G_\algebra, z^* \in \Gamma_Z\}$, where the functional  $\delta_t \otimes z^* \in X^*$ acts as follows: $(\delta_t \otimes z^*)(x) = z^*(x(t))$.
\end{theo}

\begin{proof}
Fix $\eps> 0$ and a non-empty relatively $w^*$-open subset $U \subset \Gamma$. Let $t_0\in\G_\algebra$ and $z_0^*\in\G_Z$ be   such that $\delta_{t_0}\otimes z_0^*\in U$.
Since $U$ is relatively $w^*$-open, there exist $\{x_k\}_{k=1}^n \subset X$ such that $\delta_t \otimes z^*\in\G$ belongs to $U$ whenever
\[
   \max_{1\leq k\leq n} |\langle (\delta_{t_0}\otimes z_0^*)-(\delta_t\otimes z^*), x_k\rangle| < 1.
\]
Consider the non-emtpty open set
\[
   B\coloneqq \{t\in \G_\algebra \colon |z_0^*(x_k(t))-z_0^*(x_k(t_0))|< 1\; \text{ for }1\leq k\leq n\},
\]
and define the following non-empty relatively $w^*$-open subset of $\G_Z$:
\[
   D\coloneqq \{z^* \in \G_Z: |z^*( x_k(t_0)) - z_0^*( x_k(t_0))| < 1\; \text{ for }1\leq k\leq n\}.
\]
Using property (iv) for  $ \{x_k\}_{k=1}^n \subset X$ we can find a non-empty open subset $B_1 \subset B$ and a non-empty relatively $w^*$-open subset $D_1 \subset D$ such that for every $t \in B_1$ and every $z^* \in D_1$ it holds
\[
   \max_{1\leq k\leq n}|z^*( x_k(t)) - z_0^*( x_k(t_0))| < 1.
\]
Define the non-empty subset $W\coloneqq \{\delta_t \otimes z^*: t \in B_1, z^* \in D_1\} \subset \Gamma$. It is clear that $W \subset U$.

By applying Definition~\ref{def_ACK} to $Z$, $\G_Z$, $D_1$ and $(\eps/2)$, we get $V_Z \subset D_1$, $z_1^*\in V_Z$, $e_Z \in S_{Z}$ and $F_Z \in\mathcal{O}$ satisfying (I) -- (VI). Denote also $V_{Z,1} \subset \G_Z$, the subset that appears in property (IV) (in the case of $Z \in \ACK$ we have $V_{Z,1} = \G_Z$).  By applying Lemma~\ref{lem-e-f} to $\algebra$, $\G_\algebra$, the non-empty open set $B_1$ and $(\eps/2)$, we find a non-empty subset $B_2 \subset B_1$, functions $f_0$, $e_\algebra$ (both belonging to $\algebra$) and $s_0\in B_2$,  satisfying its conclusions.

Finally,  let us define the requested {non-empty} subset $V \subset \aj{U}$ and corresponding $x_1^*\in V$, $e\in S_{X}$, $F\in L(X)$ as follows:
\begin{align*}
   V    &\coloneqq \{\delta_t\otimes z^*\colon t\in B_2, z^*\in V_Z\}\subset W \subset U,\\
   x_1^*&\coloneqq \delta_{s_0}\otimes z_1^*,\quad e(t) \coloneqq e_\algebra(t)  e_Z, \text{ for every } t\in\G_\algebra
\end{align*}
(condition (ii)  implies $e \in X$), and
\[
   (Fx)(t)\coloneqq f_0(t)F_Z(x(t)),
\]
for every $x \in X$ and for every $t \in \G_\algebra$. Conditions (i) and (iii) ensure that $F(x) \in X$. Observe that for every $x^*= \delta_t \otimes z^* \in \G$
\[
   F^*x^* = f_0(t)\left(\delta_t \otimes F_Z^* z^*\right).
\]
It remains to check the properties (I) -- (VI).
\item[(I)] It is clear that $\norm{F} =\norm{F_Z}=1$ and $\norm{Fe} = \norm{f_0e_\algebra}\norm{F_Z(e_Z)} = 1$.
\item[(II)] $x_1^*(Fe) = z_1^*(f_0(s_0)e_\algebra(s_0)F_Z(e_Z))= 1$.
\item[(III)] $F^*x_1^* = x_1^*$, since for every $x \in X$ we have
   \[
      (F^*x_1^*)(x) =  z_1^*\left(f_0(s_0)F_Zx(s_0)\right) =(F_Z^* z_1^*)(x(s_0))=  z_1^*(x(s_0))=x_1^*(x).
   \]
\item[(IV)]
   For every $x^* \in \G$, we have $x^* = \delta_t \otimes z^*$, $t\in \G_\algebra$ and $z^*\in \G_Z$. First, consider the case  $z^*\in V_{Z,1} $ and observe that
   \begin{align*}
   \norm{(I_{X^*} - F^*)(x^*)}&=\norm{z^* - f_0(t)F_Z^*z^*}  \\
      &\le|1-f_0(t)|\norm{z^*} + |f_0(t)|\cdot\norm{(I_{Z^*} - F_Z^*)(z^*)}\\
      &= |f_0(t)|\cdot \norm{(I_{Z^*} - F_Z^*)(z^*)} + |1-f_0(t)|.
   \end{align*}
   Therefore, in this case
   \begin{align*}
   &\norm{F^*x^*} + (1 - \eps)\norm{(I_{X^*} - F^*)(x^*)}  \\
	          &= |f_0(t)|\cdot\norm{F_Z^*z^*} + (1-\eps)\norm{z^* - f_0(t)F_Z^*z^*}  \\
              &\leq |f_0(t)|\bigl(\norm{F_Z^*z^*} + (1 - \eps)\norm{(I_{Z^*} - F_Z^*)(z^*)} \bigr)+ (1-\eps) |1-f_0(t)|\\
              & \le |f_0(t)| + (1-\eps) |1-f_0(t)|  \le 1.
   \end{align*}
   Whenever $Z \in \ACK$, then $V_{Z,1} = \G_Z$, so the above inequality holds for every $z^*\in \G_Z$. Thus, we have proved (IV)'. If  $Z \in \ACK_\rho$ we still must consider those  $x^*$ belonging to $\G \setminus V_1$. The above inequality implies that  $z^*\notin V_{Z,1}$ and, consequently, $ |z^*(F_Ze_Z)| \le  \rho$ which, in turn, implies that
   \[
      |x^*(Fe)| =  |f_0(t)e_\algebra(t)z^*(F_Ze_Z)| \le  \rho.
   \]

\item[(V)]
   Let $x^*= \delta_t \otimes z^* \in \G$. Recall that $F^*x^* = f_0(t)\delta_t \otimes F_Z^* z^*$. Set $V_\algebra\coloneqq \{\delta_t \colon t\in B_2\}$.  In the proof of Theorem \ref{theo:unalg->ACK} it was proved that for every $t \in \G_\algebra$ it holds
   \[
   \dist(f(t) \delta_t, \aconv\{0, V_\algebra\}) < \frac{\eps}{2}.
   \]
   On the other hand, by our construction, we deduce that
   \[
   \dist\left( F_Z^*z^*, \aconv\{0, V_Z\}\right) < \frac{\eps}{2}.
   \]
   Thus, there exist $a^* \in  \aconv\{0, V_\algebra\}$ and $b^* \in  \aconv\{0, V_Z\}$ such that
   \[
   \norm{f(t) \delta_t - a^*} < \frac{\eps}{2} \text{ and } \norm{F_Z^*z^*- b^*} < \frac{\eps}{2}.
   \]
   In particular, since $a^* \otimes b^*$ belongs to $\aconv\{0, V\}$, we can deduce that
   \begin{align*}
   \dist(F^*x^*, \aconv\{0, V\}) &\le \norm{f_0(t) \delta_t \otimes F_Z^* z^* -  a^*\otimes b^*}  \\
                                 &\le \norm{f_0(t) \delta_t \otimes F_Z^* z^* - f_0(t) \delta_t \otimes b^*} +\\
                                 &+ \norm{f_0(t) \delta_t \otimes b^* -  a^*\otimes b^*} \\
                                 & \le  \norm{ F_Z^* z^* -  b^*} + \norm{f_0(t) \delta_t  -  a^*}  < \eps.
   \end{align*}

\item[(VI)]
   For every $x^* = \delta_t\otimes z^* \in V$ we have  $t\in B_2$ and $z^*\in V_Z$. Consequently, $ |e_\algebra(t) - 1| \le  \frac{\eps}{2}$ and  $ |z^*(e_Z)-1| \le  \frac{\eps}{2}$. From this we  get
   \[
   |x^*(e) - 1| = |e_\algebra(t)z^*(e_Z)-1| = |e_\algebra(t)(z^*(e_Z) -1) + (e_\algebra(t) -1)| \le   \eps,
   \]
   which completes the proof.
\end{proof}

\begin{rem}
Under the hypothesis of the previous theorem, given $F\in L(Z)$ and $f\in\algebra$ we can consider the operators $C_{F}\colon X\to X$ and $P_f\colon X\to X$ defined, respectively, by $C_F(x)= F\circ x$ and $P_f(x)=fx$, for every $x\in X$. Then, if we set $\mathcal{O}'\coloneqq\{C_{F}\circ P_{f} : F\in\mathcal{O},\, f\in\algebra\}$, then $X$ has $\mathcal{O}'$-$\ACK_\rho$ (resp. $\mathcal{O}'$-$\ACK$) structure.
\end{rem}

Conditions (i) -- (iii) in Theorem \ref{theo:unalg->ACKvector-valued} are easily verified in concrete examples. In contrast, condition (iv) looks technical. So, in order to make Theorem \ref{theo:unalg->ACKvector-valued} more applicable, we shall present easy-to-verify sufficient conditions for (iv).

Before passing to these sufficient conditions, observe that the function of two variables $\varphi \colon \G_A \times (\G_Z, w^*) \to \K^n$ from  condition (iv) is separately continuous. Therefore, the role of sufficient condition for (iv) can be played by any theorem about quasi-continuity of a separately continuous function $f\colon U \times V \to W $. {There is a number of such theorems (see  Encyclopedia of Mathematics article \href{https://www.encyclopediaofmath.org/index.php/Separate_and_joint_continuity}{``Separate and joint continuity''} or the introduction to \cite{TBanakh})}. For example,  according to Namioka's  theorem \cite{Namioka1974} this (and a much stronger result) occurs for $U$ being a regular, strongly countably complete topological space, $V$ being a locally compact $\sigma$-compact space and $W$ being a pseudo-metric space. The results of the kind ``separate continuity implies  quasi-continuity'' that we list and apply below do not pretend to be new.

\begin{pro} \label{prop:separ-cont1}
Let $U$, $V$, $W$ be topological spaces, $V$ be discrete and $f\colon U \times V \to W $ be separately continuous. Then, $f$ is continuous (and consequently quasi-continuous).
\end{pro}
If $Z$ has property $\beta$, the corresponding $(\G_Z, w^*) $ is a discrete  topological space. Thus, the above proposition guaranties the validity of (iv) of Theorem \ref{theo:unalg->ACKvector-valued} in this case.
\begin{cor} \label{cor:separ-cont1.1}
 Under the conditions of Theorem \ref{theo:C(K,Y)}, $C(K,(Z,\tau)) \in \ACK_\rho$, where $\rho$ is the parameter from the  property $\beta$ of $Z$. If $\beta = 0$, then $C(K,(Z,\tau)) \in \ACK$. In particular, this implies the conclusion of Theorem \ref{theo:C(K,Y)}.
\end{cor}

Proposition \ref{prop:separ-cont1} also guaranties (iv) of Theorem \ref{theo:unalg->ACKvector-valued} in the case of $\G_\algebra = \N$ (just change the roles of $U$ and $V$ in Proposition \ref{prop:separ-cont1}). If we apply Theorem \ref{theo:unalg->ACKvector-valued} with $\algebra = c_0 \subset C_b(\N) = \ell_\infty$, this leads to the following result:
\begin{cor} \label{cor:separ-cont1.2}
Let $Z \in \ACK_\rho$ ($Z \in \ACK$),  $c_0(Z) \subset X \subset \ell_\infty(Z)$, and $X$ has the following property:  $(Fz_1, Fz_2, \ldots) \in X$ for every $x = (z_1, z_2, \ldots)  \in X$ and $F \in L(Z)$. Then  $X \in \ACK_\rho$ ($X \in \ACK$ respectively).
\end{cor}

This corollary is applicable to  $c_0(Z)$ and $\ell_\infty(Z)$ themselves and also for some intermediate spaces like $c_0(Z,w)$ of weakly null sequences in $Z$.

\begin{pro} \label{prop:separ-cont1.2}
Let $Z$ be a Banach space, $(\G_\algebra, \tau)$ be a topological space,  $\G_Z \subset (B_{Z^*},w^*)$, and $x_k\colon \G_\algebra \to Z$ for $k \in \{1,2,\ldots, n\}$ be $\tau$-$\sigma(Z, \G_Z)$-continuous and $\tau$-$\norm{\,\cdot\,}$-openly fragmented functions. Then, the function  {$\varphi \colon (\G_\algebra, \tau) \times (\G_Z, w^*) \to \K^n$ given} by  $\varphi(t, z^*) = (z^*(x_k(t)))_{k=1}^n $ is quasi-continuous.
\end{pro}

\begin{proof}
Fix $(t_0 , z_0^*) \in  \G_\algebra \times \G_Z$. Let $U_\algebra \subset \G_\algebra$,  $U_Z \subset \G_Z$ be open and $w^*$-open neighborhoods of $t_0$ and  $z_0^*$ respectively. Set $U \coloneqq  U_\algebra \times U_Z$. We have to show that, for a given $\eps > 0$, there exist a non-empty open subset $W_\algebra \subset U_\algebra$ and a non-empty relatively $w^*$-open subset $W_Z \subset U_Z$ such that for every $t \in W_\algebra$ and every  $z^* \in W_Z$
\begin{equation} \label{equatz^*( y_k(a^*))}
\max_{1\leq k\leq n}|z^*( x_k(t)) - z_0^*( x_k(t_0))| < \eps.
\end{equation}

Fix $\delta<\eps/4$ and define
    \[
    V_\algebra\coloneqq \left\{t \in U_\algebra \colon \max_{1\leq k\leq n}|z_0^*(x_k(t))-z_0^*(x_k(t_0))|< \delta\right\}.
    \]
The set $V_\algebra \subset U_\algebra$ is a non-emtpy open neighborhood of $t_0$ because of the $\tau$-$\sigma(Z, \G_Z)$ continuity of $x_k$ (the map $z_0^*\circ x_k$ is a $\K$-valued $\tau$-continuous function). Applying inductively the definition of openly fragmented function, we define a non-empty open set $W_\algebra \subset (V_\algebra, \tau)$ in such a way that for all $k = 1, \ldots, n$ it holds
\[
\diam (x_k(W_\algebra)) < \delta.
\]
Fix a $t_1 \in W_\algebra$ and define the non-empty relatively $w^*$-open subset  $W_Z \subset U_Z$ as follows:
\[
    W_Z\coloneqq \left\{z^* \in U_Z \colon  \max_{1\leq k\leq n}|z^*( x_k(t_1)) - z_0^*( x_k(t_1))|< \delta\right\}.
\]

Let us show, for every $t \in W_\algebra$ and every  $z^* \in W_Z$, the validity of inequality~\eqref{equatz^*( y_k(a^*))}:
 \begin{align*}
    	| z_0^*( x_k(t_0)) - z^*( x_k(t))| &\le |z_0^*(x_k(t_0))-z_0^*(x_k(t))|\\
						 &+|z_0^*(x_k(t))-z_0^*(x_k(t_1))|\\
						 &+|z_0^*(x_k(t_1))-z^*(x_k(t_1))|\\
						 &+|z^*(x_k(t_1)) - z^*(x_k(t))|.	
    \end{align*}
The first summand in the right-hand side of the previous inequality does not exceed  $\delta$ since $t \in V_\algebra$. Accordingly, the second and fourth summands are both bounded by $\delta$ since $z_0^*, z^* \in B_{Z^*}$ and $\norm{x_k(t)-x_k(t_1)}<\delta$ since $t, t_1 \in W_\algebra$ and  $\diam (x_k(W_\algebra)) < \delta$. Finally, the corresponding third summand is bounded by $\delta$ since $z^*\in W_Z$. Therefore,
$$
| z_0^*( x_k(t_0)) - z^*( x_k(t))| \le 4 \delta < \eps,
$$
which completes the proof of \eqref{equatz^*( y_k(a^*))} and that of the proposition.
\end{proof}

As an application of the previous proposition we get the following corollaries which contain as a particular case the space $C_w(K,Z)$ of $Z$-valued weakly continuous functions for $Z \in \ACK_\rho$ (or $Z \in \ACK$).

{
\begin{cor}\label{corollary_GVA}
Let $Z \in \mathcal{O}$-$\ACK_\rho$ (or $Z \in \mathcal{O}$-$\ACK$) and $\algebra\subset C(K)$  be a uniform algebra with $K$ being the space of multiplicative functionals on $\algebra$. Fix $\G_Z\subset H\subset Z^*$, where $\G_Z$ is the $1$-norming set given by the ACK structure of $Z$. Denote by $\algebra_{\sigma(Z,H)}(K,Z)$ the following subspace of $C(K,(Z,\sigma(Z,H)))$:
\[
\algebra_{\sigma(Z,H)}(K,Z) = \left\{f\in Z^K\colon z^*\circ f\in \algebra\,\text{ for all }\,z^*\in H\right\}.
\]
Let us assume that
\begin{enumerate}
\item $F^*H\subset H$ for every $F\in\mathcal{O}$.
\item $(f(K),\sigma(Z,H))$ is fragmented by the norm for every $f$ belonging to $\algebra_{\sigma(Z,H)}(K,Z)$.
\end{enumerate}
Then, $\algebra_{\sigma(Z,H)}(K,Z) \in \ACK_\rho$ (resp. $\algebra_{\sigma(Z,H)}(K,Z) \in \ACK$).
\end{cor}

\begin{proof}[Sketch of the proof:]
It relays on the use of Theorem~\ref{theo:unalg->ACKvector-valued}. Let  $\G_\algebra \subset K$ be the corresponding subset from Lemma \ref{Ur-alg}.   Then, restrictions of elements of $\algebra$ to $\G_\algebra$ form an $\ACK$-subalgebra $C_{b}(\G_A)$ isometric to $\algebra$ (that we identify with $\algebra$) and restrictions of elements of $\algebra_{\sigma(Z,H)}(K,Z)$ to $\G_\algebra$ form a subspace $X  \subset  C_{b}(\G_\algebra, (Z, {\sigma(Z,H)}))$ isometric to $\algebra_{\sigma(Z,H)}(K,Z)$.  The conditions (i) and (ii) of Theorem~\ref{theo:unalg->ACKvector-valued} follow from the definition of $\algebra_{\sigma(Z,H)}(K,Z)$. The condition (iii) of Theorem~\ref{theo:unalg->ACKvector-valued} is reduced to the present condition (i). And, finally, the condition (iv) of Theorem~\ref{theo:unalg->ACKvector-valued} is reduced to the present (ii) by using Proposition~\ref{prop:separ-cont1.2}.
\end{proof}

The condition (i) above could be quite demanding, for instance, when $\mathcal{O}=L(Z)$ in which case $H$ is forced to be $Z^*$. However, in all concrete examples that we know of $\ACK$ structure, the family $\mathcal{O}$ can be taken really small. Thus, for concrete examples of $Z$, the condition (i) could be easily satisfied for every election of $H$.

By using the results from \cite{CNV} it can be shown that condition (ii) above is satisfied for every $H$ whenever $(Z,w)$ is Lindel\"of. Indeed, given $f$ belonging to $\algebra_{\sigma(Z,H)}(K,Z)$, $f(K)\subset Z$ is $\sigma(Z,H)$-compact, thus, it is also Lindel\"of. A straightforward application of \cite[Corollary E]{CNV} ensures that $(f(K),\sigma(Z,H))$ is norm-fragmented. Hence, in this case, Corollary~\ref{corollary_GVA} can be simplified as follows:

\begin{cor}
Let $Z \in \mathcal{O}$-$\ACK_\rho$ (or $Z \in \mathcal{O}$-$\ACK$) such that $(Z,w)$ is Lindel\"of and $\algebra\subset C(K)$  be a uniform algebra with $K$ being the space of multiplicative functionals on $\algebra$. Fix $\G_Z\subset H\subset Z^*$ such that $F^*H\subset H$ for every $F\in\mathcal{O}$, where $\G_Z$ is the $1$-norming set given by the ACK structure of $Z$. Then, $\algebra_{\sigma(Z,H)}(K,Z) \in \ACK_\rho$ (resp. $\algebra_{\sigma(Z,H)}(K,Z) \in \ACK$).
\end{cor}

Observe that when $Z$ has property $\beta$, the set $\mathcal{O}$ coincides with the set $\{x_{\alpha}^*(\cdot)\,x_\alpha: \alpha\in\Lambda\}$. Therefore, in this case, $F^*H\subset H$ for every $H$ and for every $F\in\mathcal{O}$. Thus, we have proved the following corollary.
\begin{cor}
Let $Z$ be a Banach space with property $\beta$ such that $(Z,w)$ is Lindel\"of and $\algebra\subset C(K)$  be a uniform algebra with $K$ being the space of multiplicative functionals on $\algebra$. Fix $\G_Z\subset H\subset Z^*$, where $\G_Z=\{x^*_\alpha: \alpha\in\Lambda\}$. Then, $\algebra_{\sigma(Z,H)}(K,Z) \in \ACK_\rho$.
\end{cor}

However, this technique can not fully generalize Theorem~\ref{theo:C(K,Y)} by Acosta et al. to the case of vector-valued uniform algebras, since here the Lindel\"of property is essential and property $\beta$ does not imply in general weak Lindel\"of. Observe that nevertheless the original statement of Theorem~\ref{theo:C(K,Y)} is covered completely by our Corollary~\ref{cor:separ-cont1.1}.
}


\def\cprime{$'$}\def\cdprime{$''$}
  \def\polhk#1{\setbox0=\hbox{#1}{\ooalign{\hidewidth
  \lower1.5ex\hbox{`}\hidewidth\crcr\unhbox0}}}
\providecommand{\bysame}{\leavevmode\hbox to3em{\hrulefill}\thinspace}
\providecommand{\MR}{\relax\ifhmode\unskip\space\fi MR }
\providecommand{\MRhref}[2]{%
  \href{http://www.ams.org/mathscinet-getitem?mr=#1}{#2}
}
\providecommand{\href}[2]{#2}

\end{document}